\documentclass[preprint,11pt]{amsart}
 \usepackage{amssymb}
\usepackage{amsmath,amsfonts,amscd}
\usepackage{mathrsfs,hyperref, enumerate}

\newtheorem{thm}{Theorem}
\numberwithin{thm}{section}

\newtheorem{prop}[thm]{Proposition}
\newtheorem{defn}[thm]{Definition}
\newtheorem{lem}[thm]{Lemma}

\newcommand{\abs}[1]{| #1 |}
\newcommand{\trm}[1]{\textrm{#1}}
\newcommand{\al}{\alpha}
\newcommand{\tit}[1]{\textit{#1}}
\def\F{\mathbb{F}}

\newcommand{\lam}{\lambda}

\newcommand{\diag}{\textrm{diag}}

\begin{document}

\title{Matrix Waring Problem - II}

\author{Krishna Kishore}
\address{Indian Institute of Technology Tirupati, 517506, Andhra Pradesh, India}
\email{kishorekrishna@iittp.ac.in}
\thanks{The first author is partially supported by Science and Engineering Research Board (SERB) MATRICS grant MTR$/2021/000319$ of the Government of India.}
\author{Anupam Singh}
\address{Indian Institute of Science Education and Research Pune, 411008, Maharashtra, India}
\email{anupamk18@gmail.com}
\thanks{The second author is funded by SERB grant CRG/2019/000271 for this research.}

\begin{abstract}
We prove that for all integers $k \geq 1$, there exists a constant $C_k$ depending only on $k$ such that for all $q > C_k$ and for all $n \geq 1$ every matrix in $M_n(\F_q)$ is a sum of two $k$th powers.
\end{abstract}


\maketitle

\section{Introduction}

The classical Waring problem deals with expressing natural numbers as a sum of $k$th powers, where $k$ is a positive integer. Modern versions consider the same question, but over objects with non-commutative structures. For example, Shalev \cite{Sh} showed that for every finite (nonabelian) simple group of sufficiently high order every element can be expressed as values of word $w$ of length $3$. This was later improved to $2$ by Larsen, Shalev and Tiep  \cite{LST}. Larsen conjectured (in a personal communication) that a similar result should hold for matrices over finite fields. In other words, if $R$ denotes a commutative ring with unity, then the Matrix Waring Problem is to address whether matrices over $R$ can be expressed as a sum of two $k$th powers (of matrices). The goal of this article is to answer this question in the case where $R$ is a finite field $\F_q$, with $q$ sufficiently large; see Theorem \ref{intro_main_thm_1}. 

This paper is a continuation of \cite{Ki}, which uses Lang-Weil's results on the number of solutions to equations over finite fields \cite{LW} to prove that for all integers $k \geq 1$, there exists a constant $C_k$ depending only on $k$ such that for all $q > C_k$ (i) for all $n = 1, 2$ every matrix in $M_n(\F_q)$ is a sum of two $k$th powers, and (ii) for all $n \geq 3$ every matrix in $M_n(\F_q)$ is a sum of \textit{at most} three powers. In this article, we strengthen this result and prove Larsen's conjecture:

\begin{thm}\label{intro_main_thm_1}
For all integers $k \geq 1$, there exists a constant $C_k$ depending only on $k$, such that for all $q > C_k$ and for all $k \geq 1$ every matrix in $M_n(\F_q)$ is a sum of two $k$th powers.
\end{thm}

A brief description of the strategy of the proof follows. Fix $k \geq 1$. Given a matrix  $Z$ in $M_n(\F_q)$, the goal is to find matrices $A$ and $B$ in $M_n(\F_q)$ such that $Z = A^k + B^k$. This may be viewed in the algebraic-geometric way as follows. The entries of $Z$ are given in terms of homogeneous polynomials in $2 n^2$ variables of $A$ and $B$ so that the above equation defines an affine variety over $\F_q$ defined by $n^2$ polynomials in $2n^2$ variables, and the problem reduces to that of proving that the affine variety has an $\F_q$-rational point.  Let us recall Lang-Weil's theorem on the number of points of varieties over finite fields \cite{LW}. If an affine variety $X$ over $\F_q$ is geometrically irreducible and of dimension $n$ then the number of $\F_q$-rational points is given by $q^n + O(q^{(n-1)/2})$, where the implicit constant in the error term depends on the number and the degree of the equations defining $X$. Note that absolutely irreducible means that the affine variety remains irreducible when considered over the algebraic closure of $\F_q$. To solve our problem, we use a consequence of Lang-Weil's theorem found in \cite{Sc} to prove that the particular equation $X_1^{k} + \ldots + X_n^k = 1$ has a rational solution $(x_1, \ldots, x_n)$ which is `special' in the sense that the $x_i^k$ are nonzero and mutually distinct; the $k$th powers $x_i^k$  are used in constructing semisimple matrices that are $k$th powers. 

We begin by observing that if a matrix is a sum of two $k$th powers, so is any of its conjugates. Furthermore, if each of the components of a direct sum of matrices is a sum of two $k$th powers so is their direct sum. These observations allow us to apply the theory of canonical forms of matrices and reduce the problem to that of expressing special Jordan matrices $J_{\alpha,n}$ associated to a primitive root $\alpha$ of $\F_q$ and the regular nilpotent matrices $J_{0,n}$ as a sum of two $k$th powers; see notation \S $2$. While the proof of expressing $J_{\alpha,n}$ as a sum of two $k$th powers is relatively easy, the proof of expressing nilpotent Jordan block $J_{0,n}$ takes some work. 

With these results in hand we construct the matrix $A$ (in $J_{0,n} = A^k + B^k$) with $x_i^k$ as the eigenvalues in $\F_q$, so that $A$ is diagonalizable and is also a $k$th power. Based on $A$ the matrix $B$ though forced (so that the sum $A^k + B^k$ is equal to $J_{0,n}$) is not entirely so. By carefully choosing some entries of $B$ we ensure that it is also diagonalizable and is a $k$th power. Thus we prove that for every $k \geq 1$ and for every $n \geq 1$ there exists a constant $C(k,n)$ depending on \tit{both} $k$ and $n$ such that for all $q > C(k,n)$ the $J_{0,n}$ is a sum of two $k$th powers. But we want a constant depending only on $k$ and not on $n$. So, we give an independent proof that $J_{0,n}$ is a sum of two (both nilpotent this time) $k$th powers for all $n \geq 2k$. Precisely, we prove the following result which is of independent interest.

\begin{thm}\label{theorem-power-nilpotent}
Let $\mathbb F$ be a field (not necessarily finite). Let $ k \geq 2$ and $n \geq 2k $. Then, $J_{0,n}$ is a sum of two $k$th powers in $M_n(\mathbb F)$. 
\end{thm}

By the first-proof of the representibility of $J_{0,n}$ we obtain finitely many constants $C(k,1), \ldots, C(k,2k-1)$. By Theorem \ref{theorem-power-nilpotent} we obtain a constant that works for all $n \geq 2k$. Thus, for all $q$ greater than the maximum of the \tit{finitely many} constants we obtain Theorem \ref{intro_main_thm_1} as required.

As for the non-nilpotent case, it is fairly straightforward to construct matrices $A$ and $B$ such that $J_{\alpha,n} = A^k + B^k$. Indeed we prove that the primitive element $\alpha$ can be expressed as $\alpha = a_1^k + b_1^k = a_2^k + b_2^k$ such that $a_1^k \neq a_2^k$ and $b_1^k \neq b_2^k$, and use this distinction to construct $A$ and $B$ which are, essentially, a direct sum of $2 \times 2$ matrices that are diagonalizable and are $k$th powers.

The paper is organized as follows. In \S \ref{Weil_results} we obtain consequences of the Weil's results on the number of solutions to equations over finite fields. In \S \ref{reduction} we reduce the problem to that of proving that the Jordan block associated to the primitive elements and the nilpotent Jordan block is a sum of two $k$th powers. Furthermore, we give quick proof of the non-nilpotent case. Sections  \S \ref{nilpotent_case_1} and \S \ref{nilpotent_case_2} form the essential content of this article. We prove $J_{0,n}$ is a sum of two $k$th powers that includes proving Theorem \ref{theorem-power-nilpotent}. In the final section \S \ref{main_result} we prove Theorem \ref{intro_main_thm_1} assembling various results in \cite{Ki}, \cite{Sm},  and the results  of the previous sections.

\section{Notation} \label{notation}
The $n \times n$ diagonal matrix with entries $\lam_1, \ldots, \lam_n$ along the diagonal is denoted by $\diag(\lam_1, \ldots, \lam_n)$. If $A$ is an $r \times r$ matrix and $B$ is an $s \times s$ matrix, the direct sum $A \oplus B$ is the $(r + s) \times (r + s)$ block diagonal matrix  $ \begin{bmatrix} A & 0 \\ 0 & B \end{bmatrix}$. A finite field with $q$ elements is denoted by $\F_q$.

Let $F(X_1, \ldots, X_n)$ be a polynomial in $n$ indeterminates $X_1, \ldots, X_n$. An $n$-tuple $(x_1, \ldots, x_n) \in \F_q^n$ such that $F(x_1, \ldots, x_n) = 0$ is called a \textit{solution of $F$ in $\F_q^n$.}  

Let $n \geq 1$ be a positive integer. Let $\lam$ be in a field $\mathbb F$. Let $J_{\lam,n}$ denote ($J$ for Jordan) the following matrix:
\begin{equation}\label{Jordan_lambda_0}
J_{\lambda,n} = 
\begin{bmatrix}
\lam & 1 & 0 & \cdots  & 0 \\
0 & \lam & 1 & \cdots  & 0\\
\vdots & \vdots & \ddots & \ddots & \vdots \\
0 & 0 & \cdots   & \lam &1 \\
0 & 0 & \cdots    & 0& \lam 
\end{bmatrix}_{n \times n},
\end{equation}
where the diagonal entries are all $\lambda$, the super-diagonal entries are all $1$ and the remaining entries are $0$.

Let $n \geq 1$ and $r \geq 1$ be positive integers. Let $f := x^n -a_{n-1} x^{n-1}- \ldots - a_1 x - a_0 \in \F[x]$. The Jordan block $J_{f,r}$ associated to $f$ with $r$ blocks is matrix,
\begin{equation}\label{Jordan_f}
J_{f,r} :=
\begin{bmatrix}
\mathfrak C_{f} & I & 0 & \ldots & 0 &0 \\
0 & \mathfrak C_{f} & I & \ldots & 0 & 0\\
0 & 0 & \mathfrak C_{f} & \ldots & 0 & 0\\
\vdots & \vdots & \vdots & \ddots & \vdots & \vdots \\
0 & 0 & 0 & \ldots & \mathfrak C_{f} & I \\
0 & 0 & 0 & \ldots & 0 & \mathfrak C_{f}
\end{bmatrix}_{n r  \times n r}
\mathfrak C_{f} := 
\begin{bmatrix} 
0 & 0 &  \ldots & 0 & a_0  \\ 
1 & 0 &  \ldots & 0 & a_1 \\ 
0 & 1 &  \ldots & 0 & a_2 \\
\vdots &\vdots  & \ddots &   & \vdots \\
0 & 0 & \cdots& 1 &  a_{n-1}
\end{bmatrix}_{n \times n}
\end{equation}
where $\mathfrak C_f$ is the companion matrix of $f$, and $I$ is the $n \times n$ identity matrix.



\section{Reducing the Problem}\label{reduction}

We begin by reducing the problem to that of representing two special matrices as a sum of two $k$th powers. Before that, let us recall the following elementary fact. Let $R$ and $S$ be rings and $\phi \colon  R  \to S$ be a ring homomorphism. For any integer $n \geq 1$, the canonical map $\Phi \colon M_n(R) \to M_n(S)$ induced by $\phi$ as $[a_{ij}] \mapsto [\phi(a_{ij})]$ is also a ring homomorphism, and injective if $\phi$ is injective. Also, recall that for all integers $m, n \geq 1$ the ring $M_m(M_n(R))$ is canonically isomorphic to the ring $M_{mn}(R)$. 

We are interested in the special case where $R = \F_{q^n} = \F_q(\alpha)$ and $S = M_n(\F_q)$, where $\alpha$ is a primitive element of $R$.  Consider the ordered basis $(1, \alpha, \ldots, \alpha^{n-1})$ of $R$ over $\F_q$. For $\beta \in R$, the multiplication-by-$\beta$ map $ \phi_{\beta}\colon x \mapsto \beta \cdot x$ is an $\F_q$-linear map, represented in the above basis by a matrix in $S$ giving the map $\phi\colon R \to S$. In particular, if $f(t) = t^n -a_{n-1} t^{n-1}- \ldots - a_1 t - a_0$ is minimal polynomial of $\alpha$ then $\alpha$ is mapped to $\mathfrak C_f$. Together with the canonical isomorphism  $M_{nd}(\F_q) \cong M_d(M_n(\F_q))$ we obtain a canonical embedding  $M_d(R) \hookrightarrow M_{nd}(\F_q)$ that maps $J_{\alpha,d}$ to $J_{f,d}$; see \S 2 for notation.  

\begin{prop} \label{reduction}
Fix an integer $C \geq 0$.  Let $k \geq 1$ be a positive integer. For all $q > C$, $n \geq 1$ and $\alpha$ any primitive element of $\F_q$, suppose that matrices  $J_{\alpha, n}$ and $J_{0,n}$ in $M_n(\F_q)$ are a sum of two $k$th powers. Then, for all $q > C$ and $n \geq 1$, every matrix in $M_n(\F_q)$ is a sum of two $k$th powers.
\end{prop}

\begin{proof}

Let $A \in M_n(\F_q)$. Then $A$ is a sum of two $k$th powers if, and only if, for all $P \in GL_n(\F_q)$ the conjugate $P A P^{-1}$ of $A$ is a sum of two $k$th powers; see Lemma $3.1$ of \cite{Ki}. Therefore, it suffices to prove that  some, hence any, choice of representative of the conjugacy classes of $A$ is a sum of two $k$th powers. It is well known that if $f(t)$ is the characteristic polynomial of $A$ with the factorization $\prod_{i =1}^k f_i(t)^{r_i}$, where $r_i \geq 1$ and $f_i(t)$ are irreducible polynomials over $\F_q$, then $A$ is conjugate to the block diagonal matrix $J_{f_1, r_1} \oplus  \ldots \oplus J_{f_k,r_k}$. Clearly, if each Jordan block $J_{f_i, r_i}$ is a sum of two $k$th powers, then so is their direct sum. Therefore, it suffices to prove that matrices of the form $J_{f,r}$ is a sum of two $k$th powers, where $r \geq 1$ and $f$ is, by abuse of notation, an irreducible polynomial over $\F_q$. If $f(t) = t$ with zero constant term then $J_{f, r}$ is the regular nilpotent matrix of size $r$, and it is a sum of two $k$th powers by hypothesis. On the other hand, by the preceding discussion before the theorem, the matrix $J_{f, r}$ in the case where $f(t) \neq t$ is irreducible over $\F_q$ and $r \geq 1$,  it is a sum of two $k$th powers if $J_{\alpha,r}$ is a sum of two $k$th powers, which indeed is the case by hypothesis.
\end{proof}

Thus we are reduced to showing that $n \times n$ matrices $J_{\alpha,n}$, $J_{0,n}$  are a sum of two $k$th powers. While the former case is almost immediate, it is the nilpotent case that takes some work.

First we prove that the Jordan block of the form $J_{\alpha,n}$ where $\al$ is a primitive element of $\F_q$ is a sum of two $k$th powers, for all sufficiently large $q$. The reader may notice that the following proposition holds for all $n \geq 2$, in particular including $n = 2$. So the result includes some of the results in \cite{Ki}. But not everything found in \cite{Ki} is subsumed by the results in this article. We will see that we still need results in \cite{Ki} in the characteristic $2$ case.

\begin{prop}\label{primitive_element}
For every integer $k \geq 1$ there exists a constant $C_k$ depending only on $k$ such that for all $q > C_k$ and $n \geq 2$ every matrix $J_{\alpha,n}$ is a sum of two $k$th powers,  where $\alpha$ is a primitive element of $\F_q$.
\end{prop}
\begin{proof}

By Proposition \ref{non_nilpotent_case}, given $k \geq 2$ there exists a constant $C_k$ depending only on $k$ such that for all $q > C_k$ there exists pairs $(a,b)$ and $(c,d)$ in $\F_q \times \F_q$ such that $a^k \neq c^k$ and $b^k \neq d^k$ and $a^k + b^k = c^k + d^k = \alpha$. For these two $2$-tuples we define the following matrices. 

For $n \geq 2$ even, let 
$$
G_n = \bigoplus_{n /2} \begin{bmatrix} a^k &  1\\ 0 & c^k \end{bmatrix} \trm{ and } H_n = [b^k] \oplus \left( \bigoplus_{(n-2)/2} \begin{bmatrix} d^k & 1 \\ 0 &b^k \end{bmatrix} \right) \oplus [d^k].
$$
For $n \geq 3$ odd, let 
$$
G_n = \left( \bigoplus_{(n-1)/2} \begin{bmatrix} a^k & 1 \\ 0 & c^k \end{bmatrix} \right) \oplus [a^k] \trm{ and } H_n = [b^k] \bigoplus \left( \bigoplus_{(n-1)/2} \begin{bmatrix} d^k & 1 \\ 0 & b^k \end{bmatrix} \right).
$$
For example, the matrices $G_5$ and $H_5$ corresponding to $n = 5$ have the following form: 
$$
G_5 = 
\begin{bmatrix}
a^k & 1 &  &  &  \\
  & c^k  &  &   &  \\
& & a^k & 1 &  \\
& & & c^k &  \\
& & & & a^k \\
\end{bmatrix}
\textrm{ and }
H_5 = 
\begin{bmatrix}
b^k &  & & & \\
  & d^k  &  1& &  \\
& &   b^k &   & \\
& & &  d^k & 1\\
& & &  & b^k \\
\end{bmatrix}.
$$
Clearly $\begin{bmatrix} a^k & 1 \\  & c^k \end{bmatrix}$ is diagonalizable because it has distinct eigenvalues $a^k$ and $c^k$, so it is conjugate to $\begin{bmatrix} a^k &  \\  & c^k \end{bmatrix}$, visibly a $k$th power. Similarly $\begin{bmatrix} d^k & 1 \\  & b^k \end{bmatrix}$  is diagonalizable because it has distinct eigenvalues $d^k$ and $b^k$, so it is conjugate to a diagonal matrix $\begin{bmatrix} d^k &  \\  & b^k \end{bmatrix}$ which is a $k$th power too. Since representation of a matrix as a sum of two $k$th powers is stable under direct sum of matrices, it follows that $G_n$ and $H_n$ are $k$th powers and therefore their sum is represented as a sum of two $k$th powers for all $q > C_k$.
\end{proof}
We consider the remaining and more difficult case in the next section, namely that of representing a nilpotent matrix as a sum of two $k$th powers. 

\section{Nilpotent Case- I}\label{nilpotent_case_1}

An $n \times n$ matrix with $0$'s at or below the main diagonal, and all nonzeros in the diagonal immediately above the main diagonal is conjugate to $J_{0,n}$. Therefore it suffices to show that  $J_{0,n}$ is a sum of two $k$th powers. The following result is a \tit{weaker} one, in the sense that the constant given by this result depends on \textit{both} $n$ and $k$, while we desire to have a constant depending only on $k$ which we will achieve in the next section.

\begin{prop}\label{C_depend_k_n}
For every integer $k \geq 1$ and for every $n \geq 3$ there exists a constant $C(k,n)$ depending on both $k$ and $n$ such that for all $q > C(k,n)$ with the characteristic of $\F_q$ not $2$, the matrix $J_{0,n}$ is a sum of two $k$th powers where at least one of them is diagonalizable.
\end{prop}

\begin{proof}

For a given $k \geq 1$ and $n \geq 3$ we begin with considering matrices $A_n$ and $B_n$ of the following form. Let $y_1, \ldots, y_{n-1}$ and $z_1, \ldots, z_{n-1}$ be some indeterminate variables that will be assigned value in $\F_q$ later in the proof. Let $A_n$ be the matrix whose last row is $y=[y_1 \; y_2 \; \ldots \; y_{n-1} \; 1]$, with $1$'s immediately above the diagonal (also called super-diagonal) and $0$s everywhere else; that is $A_n$ is of the form
$$
\begin{bmatrix}
0 & \vline & I_{n-1} \\
\hline
 & y & \\
\end{bmatrix}.
$$

\noindent Let $B_n$ be the matrix with the last row $[-y_1 \; -y_2 \; \ldots \; -1 ]$ and the last column $[z_1 \; z_2 \; \ldots \;  z_{n-2} \; 0 \; -1]^T$ and with $1$s along the super-diagonal except at the $((n-1), n)$-th position; that is $B_n$ is of the form
$$
\begin{bmatrix}
0 & \vline & I_{n-2} & \vline & z \\
\hline 
& & 0 & & \\
\hline
 & & -y & \\
\end{bmatrix}.
$$

\noindent The sum of $A_n$ and $B_n$ is the matrix with $[z_1 \; z_2 \; \ldots z_{n-2} \; 1 \; 0]^T$ as the last column, and entries $2, 2, \ldots, 2,  1$ along the super-diagonal and $0$'s everywhere else. By hypothesis the characteristic of $\F_q$ is not $2$, so $A_n + B_n$ is conjugate to $J_{0,n}$.  Therefore, it suffices to show that $A_n$ and $B_n$ are $k$th powers for appropriate choices of $y_i$ and $z_j$.

In constructing the matrices $A_n$ and $B_n$ we want them each to have $n$ distinct eigenvalues, each of which is a $k$th power in $\F_q$. Since such a matrix is semisimple, it is a $k$th power. First we show that $A_n$ is a $k$th power for an appropriate choices of $y_i$. Note that $A_n$ is the transpose of the companion matrix of the characteristic polynomial $t^n - t^{n-1} - y_{n-1} t^{n-2} - y_{n-2} t^{n-3} - \ldots - y_2 t - y_1$, so its characteristic polynomial is the same. By Proposition \ref{weil_nilpotent_case}, there exists a solution $(\alpha_1, \ldots, \alpha_n) \in \F_q^n$ to $X_1^k + \ldots + X_n^k = 1$ such that $\alpha_i \neq 0$ for all $1 \leq i \leq n$ and $\alpha_i^k \neq \alpha_j^k$ for all $1 \leq i, j \leq n$ and $i \neq j$. Note that the requirement $\sum_{i =1}^n X_i^k = 1$  corresponds to the condition that we want to set of coefficient of $t^{n-1}$ in the characteristic polynomial of $A_n$ to be $-1$. We want the characteristic polynomial of $A_n$ be equal to the polynomial
$$
(t - \alpha_1^k) \cdots (t - \alpha_n^k)
$$
so that we \textit{have} to define $y_j := (-1)^{n - j} S_{n-j+1}(\alpha_1^k, \ldots, \alpha_n^k)$, the $j$-th elementary symmetric polynomial in $\alpha_j^k$. Running the argument backwards, we see that with these choices of $y_j$, the characteristic polynomial of $A_n$ is
$$
t^n - t^{n-1} - y_{n-1} t^{n-2} - \ldots - y_1  = (t - \alpha_1^k) \cdots (t - \alpha_n^k).
$$
On the other hand, since $\alpha_i^k \neq \alpha_j^k$ for all $i \leq j$,  it follows that $A_n$ has $n$ mutually distinct eigenvalues in $\F_q$, so it is conjugate to the diagonal matrix $\diag(\alpha_1^k, \ldots, \alpha_n^k)$ which is visibly a $k$th power. Therefore $A_n$, being conjugate to a $k$th power, is a $k$th power.

Now we show that $B_n$ is a $k$th power for an appropriate choice of $z_i$s. By Lemma \ref{ele_com}, the characteristic polynomial of $B_n$ is 
\begin{align*}
&t^n + t^{n-1} \\
&+ (y_1 z_1 + y_2 z_2 + \ldots + y_{n-2} z_{n-2}) t^{n-2}  \\
&+ (y_1 z_2 + y_2 z_3 + \ldots + y_{n-3} z_{n-2}) t^{n-3} \\
&+ \ldots \\
&+ (y_1 z_{n-3}  + y_2 z_{n-2} ) t^{2} \\
&+ (y_1 z_{n-2} ) t.
\end{align*}
For the above choice of $k$ and $n$, again by Proposition \ref{weil_nilpotent_case} there exists a constant $C(n-1,k)$ depending on both $n$ and $k$ such that for all $q > C(n-1,k)$ there exists an $(n-1)$-tuple $(\beta_1, \ldots, \beta_{n-1}) \in \F_q^{n-1}$ such that $(\beta_1, \ldots, \beta_{n-1})$ is a solution to 
$$
X_1^k + \ldots + X_{n-1}^k = -1.
$$ 
Again, note that this requirement corresponds to the condition that we want to set of coefficient of $t^{n-1}$ in the characteristic polynomial of $B_n$ to be $1$. Moreover, $\beta_i^k \neq \beta_j^k$ for all $1 \leq i, j \leq n-1$ and $i \neq j$, and $\beta_i \neq 0$ for all $1 \leq i \leq n-1$. Thus we have a solution $(0, \beta_1, \ldots, \beta_{n-1}) \in \F_q^n$ such that the $k$th powers of the coordinates are mutually distinct and $\beta_i$ are nonzero. 
Now, we force the characteristic polynomial of $B_n$ to be 
$
t (t - \beta_1^k) \cdots (t - \beta_{n-1}^k)
$
as follows.
Let $s_i$ denote the $i$th elementary symmetric polynomial in $\beta_1^k, \ldots, \beta_{n-1}^k$. Equating the coefficients, we obtain the following relations
\begin{align*}
y_1 z_1 + y_2 z_2 +y_3 z_3  \ldots  + y_{n-2} z_{n-2} &= s_2 \\
y_1 z_2 + y_2 z_3 + \ldots + y_{n-2} z_{n-2}  & =- s_3 \\
\vdots & \\
y_1 z_{n-3} + y_2 z_{n-2} &= (-1)^{n-2} s_{n-2} \\
y_1 z_{n-2} &= (-1)^{n-1}s_{n-1}.
\end{align*}
The coefficient matrix of the linear system above has nonzero determinant, in fact equal to $y_1^{n-1}$ (here is where we use that $y_1$ is nonzero), so it has a unique solution $[z_1 \; \ldots \; z_{n-1}]^{T}$. With these choices of $z_i$, it follows that the characteristic polynomial of $B_n$ is equal to $t(t - \beta_1^k) \cdots (t - \beta_{n-1}^k)$. Since the $k$th powers of $\beta_i$ are mutually distinct and nonzero it follows that, including the $0$ as an eigenvalue, $B_n$ has $n$ mutually distinct eigenvalues in $\F_q$, so it  is conjugate to the diagonal matrix $\diag(0, \beta_1^k, \ldots, \beta_{n-1}^k)$ which is visibly a $k$th power. 
\end{proof}

We reiterate for the sake of clarity that the above Proposition \ref{C_depend_k_n} is distinct from Theorem \ref{intro_main_thm_1}; the constant $C(k,n)$ depends on both $k$ and $n$, while the constant $C_k$ in Theorem \ref{intro_main_thm_1} depends only on $k$. We now prove that for all $n$ greater than a constant depending on $k$,  the  $J_{0,n}$ is a sum of two $k$th powers. The proof is independent of Proposition \ref{C_depend_k_n}. So we are not actually giving two proofs of the same result, the constant in Proposition \ref{C_depend_k_n} depends on $k$ and $n$ while the following result is one about representibility of $J_{0,n}$ for all $n \geq 2k$. Based on these two results we obtain a constant that is independent of $n$ but depends only on $k$.

\section{Nilpotent Case-II}\label{nilpotent_case_2}

Results in this section are valid for all fields, not necessarily finite. Let $\mathbb F$ be any field and  $M_n(\mathbb F)$ be the ring of $n\times n$ matrices with entries in $\mathbb F$. Let $k \geq 2$ be a positive integer. Our goal is to express $J_{0,n}$ as a sum of two $k$th powers. In contrast with Proposition \ref{C_depend_k_n} which expresses the regular nilpotent matrix $J_{0,n}$ as a sum of $k$th powers of \textit{semisimple} elements of $M_{n}(\F_q)$, the main result of this section expresses $J_{0,n}$ as a sum of two $k$th powers of \textit{nilpotent} matrices. By Proposition \ref{reduction} it suffices to consider the regular nilpotent matrix $J_{0,n}$.

For a real number let $\lfloor x \rfloor$ denote the largest integer less than or equal to $x $ and $\lceil x\rceil$ denotes the smallest integer greater than or equal to $x$. Consider Lemma $8.1$ of \cite{KS} attributed to Miller and quoted here for sake of clarity.
\begin{lem}[Miller]\label{miller}
Let $n,k \geq 1$ be positive integers, and suppose $n > k$. Let $m$ be $n$ congruent modulo $k$ with $0 \leq m \leq k-1$. Then the $k$-th power $J_{0,n}^k$ is conjugate to 
$$\left( \bigoplus_{k-m} J_{0,\lfloor\frac{n}{k}\rfloor} \right) \bigoplus \left( \bigoplus_{m} J_{0,\lceil\frac{n}{k}\rceil}\right).$$
\end{lem}

\begin{defn} 
For a positive integer $n \geq 1$, the $r$-tuple $(n_1, \ldots, n_r)$, for $r \geq 1$, is a partition of $n$ if $1 \leq n_1 \leq \ldots \leq n_r$ and $n = n_1 + \ldots + n_r$. The coordinates $n_i$ are called {\normalfont parts of the partition}. The {\normalfont Junction matrix} associated to the partition $(n_1, \ldots, n_r)$ is the matrix $$
\mathfrak J_{(n_1, \ldots, n_r)} = e_{n_1, n_1+1}+ e_{(n_1+n_2), (n_1+n_2+1)}+\cdots + e_{(n_1+\ldots+n_{r-1}), (n_1+\ldots+n_{r-1}+1)}
$$
where $e_{i,j}$ denotes the matrix with $1$ at ij-th place and $0$ elsewhere.
\end{defn}
For example, corresponding to the partition $(1,1,2,2)$ and $(2,2,2)$ of $6$, the junction matrices are
$$\newcommand*{\tempz}{\multicolumn{1}{c|}{0}}
\left[\begin{array}{cccccc}
\tempz & 1 & 0 & 0 &0 &0  \\ \cline{1-2}
\tempz & \tempz &1 & 0 &0 &0 \\ \cline{2-4}
0&\tempz&0&\tempz &0&0\\
0 & \tempz & 0& \tempz & 1 &0 \\ \cline{3-6}
0 & 0& 0& \tempz & 0&0 \\
0 & 0& 0& \tempz& 0&0 
\end{array}\right], \left[\begin{array}{cccccc}
0&\tempz & 0 & 0 &0 &0  \\ 
0 & \tempz &1 & 0 &0 &0 \\ \cline{1-4}
0&\tempz&0&\tempz &0&0\\
0 & \tempz & 0& \tempz & 1 &0 \\ \cline{3-6}
0 & 0& 0& \tempz & 0&0 \\
0 & 0& 0& \tempz& 0&0 
\end{array}\right] .$$

Clearly the junction matrices are nilpotent. When these are $k$th power is determined as follows: 
\begin{lem}\label{Lemma-junction-conjugate}
For a positive integer $k \geq 1$, let $(n_1, \ldots, n_k)$ be a partition of $n$. Suppose $n \geq 2k$ and $n_i \geq 2$ for all $1 \leq i \leq r$. Then $\mathfrak J_{(n_1, \ldots, n_k)}$ is a $k$th power. 
\end{lem}
\begin{proof} 
The junction matrix $\mathfrak J_{(n_1, n_2, \ldots, n_k)}$ associated to the partition $(n_1, \ldots, n_k)$ may be viewed as the matrix of the endomorphism $T$ of $\mathbb F^n$ in the standard basis, say $(e_1, \ldots, e_n)$; in particular, it maps $e_{(n_1+\ldots + n_i)+1}$ to $e_{n_1+\ldots + n_i}$ for all $1 \leq i \leq k-1$ and maps other basis vectors to $0$. Consider the following basis obtained by reordering $(e_1, \ldots, e_n)$:
\begin{eqnarray*}
&(e_1,\ldots, e_{n_1-1}, e_{n_1+2}, \ldots, e_{(n_1+n_2-1)}, e_{(n_1+n_2+2)}, \ldots, \\ & e_{(n_1+\ldots +n_{k-1})-1}, e_{(n_1+\ldots+n_{k-1})+2},  \ldots, e_{(n_1+n_2+\ldots + n_k)}, \\&  
 e_{n_1},  e_{n_1+1}, e_{n_1+n_2}, e_{(n_1+n_2 +1)} \ldots, e_{(n_1+\ldots +n_{k-1})}, e_{(n_1+\ldots+n_{k-1})+1}).
\end{eqnarray*}
Here we put the tuple of basis vectors $e_{n_1+\ldots +n_{i}}, e_{n_1+\ldots+n_{i}+1}$ for $1\leq i \leq k-1$ at the end (since $n_i\geq 2$) and all others in the beginning. The matrix of $T$ in this basis is
$$
A:=\displaystyle \left( \bigoplus_{n-2(k-1)} J_{0,1} \right) \bigoplus \left(\bigoplus_{k-1} J_{0,2} \right),
$$
where $J_{0,1}$ is the $1 \times 1$ matrix $[0]$.
Now we use the hypothesis $n \geq 2k$; let 
$$
B :=  \left( \bigoplus_{n-(2k-1)} J_{0,1} \right)  \bigoplus  J_{0,2k-1}.
$$
By Lemma~\ref{miller}, $B^k$ is conjugate to $A$ which in turn is conjugate to $\mathfrak J_{(n_1, n_2, \ldots, n_k)}$ so it is a $k$th power. 
\end{proof}
The condition $n_i \geq 2$ in Lemma \ref{Lemma-junction-conjugate} is useful. Indeed, if some $n_i < 2$ then junction matrix has a Jordan block of size $> 2$. For example, $\mathfrak J_{(1,1,1,3)}=$ 
$$\newcommand*{\tempz}{\multicolumn{1}{c|}{0}}
\left[\begin{array}{cccccc}
\tempz & 1 & 0 & 0 &0 &0  \\ \cline{1-2}
\tempz & \tempz &1 & 0 &0 &0 \\ \cline{2-3}
0&\tempz&\tempz& 1 &0&0\\ \cline{3-6} 
0 & 0 & \tempz& 0& 0 &0 \\ 
0 & 0& \tempz& 0 & 0&0 \\
0 & 0& \tempz& 0& 0&0 
\end{array}\right], A= \left[\begin{array}{cccccc}
0&\tempz & 0 & 0 &0 &0  \\ 
0 & \tempz &1 & 0 &0 &0 \\ \cline{1-6}
0&\tempz&0&1 &0&0\\
0 & \tempz & 0& 0 & 1 &0 \\ 
0 & \tempz& 0& 0 & 0&1 \\
0 & \tempz& 0& 0 & 0&0 
\end{array}\right] = J_{0,1}\oplus J_{0,1}\oplus J_{0,4}$$
where $A$ is conjugate of $\mathfrak J_{(1,1,1,3)}$ with respect to the ordered basis $(e_5, e_6, e_1,e_2,e_3,e_4)$. Getting such matrices as a $k$th power may not be always possible. Another illustration of the above result, the junction matrix associated to the partition $(2,2,2,3)$ and its conjugate $A$ (where $A$ is as defined in the proof) are $\mathfrak J_{(2,2,2,3)}=$ 
$$ \newcommand*{\tempz}{\multicolumn{1}{c|}{0}}
\newcommand*{\tempo}{\multicolumn{1}{c|}{1}}
\left[\begin{array}{ccccccccc}
0&\tempz  & 0 & 0 &0 &0 &0&0&0 \\ 
0&\tempz  &1 & 0 &0 &0 &0&0&0\\ \cline{1-4}
0&\tempz&0&\tempz &0&0&0&0&0\\
0 & \tempz & 0& \tempz & 1 &0&0&0&0 \\ \cline{3-6}
0&0&0&\tempz &0&\tempz&0&0&0\\
0 & 0 & 0& \tempz & 0 &\tempz&1&0&0 \\ \cline{5-9}
0 & 0& 0& 0& 0&\tempz&0&0&0\\
0 & 0& 0& 0& 0&\tempz&0&0&0\\
0 & 0& 0& 0& 0&\tempz &0&0&0
\end{array}\right], A= \left[\begin{array}{ccccccccc}
0&0  & \tempz & 0 &0 &0 &0&0&0 \\ 
0&0  &\tempz & 0 &0 &0 &0&0&0\\ 
0&0&\tempz & 0&0&0&0&0&0\\\cline{1-5}
0 & 0 & \tempz&0 & \tempo &0&0&0&0 \\ 
0&0&\tempz &0&\tempz&0&0&0&0\\\cline{4-7}
0 & 0 & 0& 0 & \tempz &0&\tempo&0&0 \\ 
0 & 0& 0& 0& \tempz &0&\tempz&0&0\\\cline{6-9}
0 & 0& 0& 0& 0&0&\tempz &0&1\\
0 & 0& 0& 0& 0&0&\tempz&0&0
\end{array}\right].$$

Now we prove the main theorem of this section.

\begin{proof}
(of Theorem \ref{theorem-power-nilpotent}) We are given $k\geq 2$ and $n\geq 2k$. Let $m$ be the integer such that $ m \equiv n \pmod{k}$ and $0 \leq m < k$. Let 
$$
A :=  \left( \bigoplus_{k-m} J_{0,\lfloor\frac{n}{k}\rfloor} \right) \bigoplus  \left( \bigoplus_{m} J_{0,\lceil\frac{n}{k}\rceil} \right).
$$
It follows from Lemma \ref{miller} that $A$ is conjugate to  $J_{0,n}^k$, so $A$ is an $k$th power. Now we observe that $J_{0,n} = A + \  \mathfrak J$, where $\mathfrak J$ is a junction matrix associated to the following partition of $n$: 
$$\left( \underbrace{\lfloor\frac{n}{k}\rfloor, \ldots, \lfloor\frac{n}{k}\rfloor}_{k-m \trm{ terms }} , 
\underbrace{\lceil\frac{n}{k}\rceil, \ldots , \lceil\frac{n}{k}\rceil}_{m \trm{ terms }} \right).$$ 
Note that the parts of the partition $\lfloor\frac{n}{k}\rfloor$ and $\lceil\frac{n}{k}\rceil$ are at least $2$ since $n\geq 2k$, so by  Lemma~\ref{Lemma-junction-conjugate}, $\mathfrak J$ is an $k$th power too. Therefore $J_{0,n}$ is a sum of two $k$th powers.
\end{proof}

\section{Main Result}\label{main_result}

In this section we assemble results of 
\cite{Sm} and \cite{Ki} and various results in the previous section to prove the main result of this article.

\begin{thm}
For all integers $k \geq 1$, there exists a constant $C_k$ depending only on $k$ such that for all $q > C_k$ and for all $n \geq 1$ every matrix in $M_n(\F_q)$ is a sum of two $k$th powers.
\end{thm}

\begin{proof}
We consider various cases. The symbol $p$ below denotes the characteristic of $\F_q$. \\

\noindent \textbf{(Case $k \geq 1$, $n =1$):} The result is due to Small \cite{Sm}. The constant $C_1$ obtained depends only on $k$, in fact it is $k^4$.\\

\noindent \textbf{(Case $k \geq 1$, $n =2$):} The result is due to Kishore; see Theorem $1.1$ of \cite{Ki}. The constant $C_{2}$ in this case depends only on $k$. \\

\noindent \textbf{(Case $k \geq 1$, $n \geq 3$, $p = 2$):} The result is due to Kishore;  see Theorem $1.2$ of \cite{Ki}  and the remark immediately after its proof [page 93. \cite{Ki}]. Indeed, the proof of Theorem $1.2	$ in \cite{Ki} shows that for $q$ sufficiently large, if $-1$ is a $k$th power then every $n \times n$ matrix over $\F_q$ is the sum of two $k$th powers. Since this is true for $p=2$, this case follows. The constant $C_{3}$ depends only on $k$; indeed, the proof of Theorem $1.2$ immediately implies that the constant is in fact the one given by the more special case $n =2$ above.\\

\noindent \textbf{(Case $k \geq 1$, $n \geq 3$, $p \neq 2$):} By Proposition \ref{reduction} with the choice of constant $C = 0$, it suffices to restrict our attention  to matrices of the form $J_{\alpha,n}$ where $\alpha$ is a primitive element of $\F_q$, and also to the regular nilpotent square matrix of size $n$. By Proposition \ref{primitive_element}, which is valid for field of characteristic not $2$, it follows that there exists a constant $C_{4}$ that depends only on $k$ such that for all $q > C_{4}$ the matrix $J_{\al,n}$ is a sum of two $k$th powers. On the other hand, by Theorem \ref{theorem-power-nilpotent}, for every $n \geq 2k$, every $J_{0,n}$ is a sum of two $k$th powers, so there are only finitely many exceptions, namely those positive integers $n$ strictly less than $2k$. But, by Theorem \ref{C_depend_k_n} for each such $n$ there exists a respective constant $C(k,n)$ such that $J_{0,n}$ is a sum of two $k$th powers. Let $C_5$ be the maximum of the constants
$$
C_5 = \trm{max}(C_4, C(k,1), C(k,2), \ldots, C(k,2k-1)).
$$
Together with Theorem \ref{theorem-power-nilpotent}, it follows that  for all $q > \textrm{max}(C_5,2k)$, the Jordan nilpotent matrix $J_{0,n}$ is a sum of two $k$th powers as desired.

Let $C_6$ be the maximum of $C_1, C_2, C_3, C_5$ and $2k$; $C_6$ depends only of $k$. Then for $q > C_6$ and for all $n \geq 1$ every matrix is a sum of two $k$th powers. 

\end{proof}

\appendix
\section{Equations over Finite Fields}\label{Weil_results}

Let $k \geq 1$ be a positive integer. In this section we prove that the equation $X_1^k + \ldots + X_n^k = \lambda$, where $\lambda$ is a nonzero element in $\F_q$, has sufficiently many $\F_q$-rational solutions, especially those solutions $(x_1, \ldots, x_n)$ whose $k$th powers $x_i^k$ are nonzero and mutually distinct. Eventually $x_i^k$ serve as eigenvalues of matrices we construct later.

We use Lang-Weil's result on the number of solutions to equations over finite fields \cite{LW}. For an introduction to solutions of equations over finite fields the reader may refer to \cite{Sc}. 
\begin{thm}\label{weil_thm} (Weil)
Let $k > 0$ be a positive integer, and consider the following polynomial in $\F_q[X_1, \ldots, X_n]$:
$$
F(X_1, \ldots, X_n) := a_1X_1^{k} +  a_2 X_2^{k} + \ldots + a_n X_n^{k}- 1,
$$
where $a_i \neq 0$ for all $1 \leq i \leq n$. Let $N$ be the number of solutions  in $\F_q^n$ of the equation $F(X_1, \ldots, X_n) = 0$. Then 
$$
\abs{N - q^{n-1}} \leq k^n \sqrt{q^{n-1}} \left( \frac{q}{q-1} \right)^{n/2}.
$$
Since $q/(q-1) \leq 2$, it follows that if $k \geq 2$ then 
\begin{equation}\label{weil_inequality}
\abs{N - q^{n-1}} \leq k^{2n} \sqrt{q^{n-1}}.
\end{equation}
\end{thm}

\begin{proof}
For a proof the reader may refer to Theorem $5$A [page 160, \cite{Sc}]. Note that we modified the statement of Theorem $5$A to suit our needs; this modified statement is an immediate and a trivial consequence of Theorem $5$A. 
For the last part, note that 
$$
k^n \sqrt{q^{n-1}} \left( \frac{q}{q-1} \right)^{n/2} \leq k^n \sqrt{q^{n-1}} (\sqrt{2})^{n} \leq k^n \sqrt{q^{n-1}} k^n \leq k^{2n} \sqrt{q^{n-1}}.
$$
\end{proof}

Our goal is to find a solution $(x_1, \ldots, x_n)$ in $\F_q^n$ to $F$ such that the $k$th powers $x_i^k$ of the coordinates $x_i$ are  nonzero and mutually distinct.

\begin{prop}\label{weil_nilpotent_case}
Let $\lambda$ be any nonzero element in $\F_q$. Fix positive integers $n \geq 3$ and $k \geq 1$, and consider the polynomial 
$$
F(X_1, \ldots, X_n) := X_1^{k} + X_2^{k} + \ldots + X_n^k - \lambda \in  \F_q[X_1, \ldots, X_n].
$$
There exists a constant $C(k,n) :=$ {\normalfont max} $(2n^2, k^{4n/(n-8)})$ (which depends on both $k$ and $n$) such that for all $q >  C(k,n)$ there exists a solution $(x_1, \ldots, x_n)$ to $F = 0$ in $\F_q^n$ such that  
\begin{enumerate}[(a)]
\item \label{nonzero}
$x_i \neq 0$, for all $1 \leq i \leq n$;
\item \label{mutually_distinct}
$x_i^k \neq x_j^k$, for all $1 \leq i , j \leq n$ and $i \neq j$. 

\end{enumerate}

\end{prop}

\begin{proof} Let $p$ be the characteristic of $\F_q$. Let $k = l \cdot p$ for some integer $l \geq 0$. The $p$-th power map on $\F_q$ is an automorphism of $\F_q$, so the equation $X_1^l + \ldots + X_n^l - \lambda = 0$ has a solution satisfying the conditions \eqref{nonzero} and \eqref{mutually_distinct} if and only if its $p$-th power  
$$
(X_1^l + \ldots + X_n^l - \lambda)^p = X_1^k + X_2^k + \ldots + X_n^k - \lambda^p,
$$
has a solution satisfying the same conditions. Hence we may assume that $p \nmid k$.  
Let $N_0$ be the number of solutions $(x_1, \ldots, x_n)$ in $\F_q^n$  to  the equation $F(X_1, \ldots, X_n) = 0$. By Theorem \ref{weil_thm}, 
$$
q^{n-1} - C_0 \leq N_0 \leq q^{n-1} + C_0, \trm{ where }C_0 = k^{2n} \sqrt{q^{n-1}} .
$$

Let $N_1$ be the number of solutions $(x_1, \ldots, x_n)$ to the equation $F=0$  in $\F_q^n$ with $
x_{i}^k = x_{j}^k $ for some $i,j$. Then $(x_1, \ldots , x_i, \ldots, \widehat{x_j}, \ldots x_n)$ is a solution to the equation 
\begin{align*}
&X_{1}^k + \ldots  + 2 X_{i}^k + \ldots + \widehat{X_{j}^k} + \ldots + X_n^k = 1.
\end{align*}
where the $\widehat{\cdot}$ denotes that the term is dropped. By hypothesis that the characteristic of $\F_q$ is not $2$, so the above equation is an equation in $n-1$ indeterminates $X_1, \ldots, X_i, \ldots, \widehat{X_j}, \ldots, X_n$. Taking into account that there are ${n \choose 2}$ such pairs $(i,j)$, it follows from Theorem \ref{weil_thm} that $N_1$ satisfies the following inequality:

$$
{n \choose 2} \left(q^{n-2} - C_1 \right) \leq N_1  \leq {n \choose 2} \left( q^{n-2} +  C_1 \right), 
$$
where $C_1 = k^{2(n-1)} \sqrt{q^{n-2}} \leq k^{2n} \sqrt{q^{n-1}}$.
Similarly, let $N_2$ be the number of solutions $(x_1, \ldots, x_n)$ to the equation $F(X_1, \ldots, X_n)=0$ in $\F_q^n$ with $
x_i = 0 $  for some $1 \leq i \leq n$ . Then 
$(x_1, \ldots, \widehat{x_i}, \ldots x_n)$,
is a solution to the equation 
\begin{align*}
&X_{1}^k + \ldots  +  X_{i-1}^k + X_{i+1}^k + \ldots + X_n^k = 1.
\end{align*}
Again by Theorem \ref{weil_thm}, $N_2$ satisfies the inequality 
$$
{n \choose 1} \left( q^{n-2} - C_1 \right) \leq  N_2 \leq {n \choose 1}  \left(  q^{n-2} + C_1 \right)
$$
where $C_1$ is as above. Note that $N_0 - N_1 - N_2$ is \textit{not} the exact number of solutions satisfying the desired conditions, because there are solutions counted in $N_1$ and $N_2$, but the exact number of solutions is at least $N_0 - N_1 - N_2$. Therefore, we may find a lower bound for $N_0 - N_1 - N_2$, and one such bound is given by 

\begin{align*}
&\left(q^{n-1} - C_0 \right) - {n \choose 2} \left(q^{n-2} + C_1 \right) - {n \choose 1}\left( q^{n-2} + C_1 \right)\\
&= \left( q^{n-1} -  {n +1 \choose 2} q^{n-2} \right)  - \left( C_0 +{n +1 \choose 2}  C_1 \right) \\
& = q^{n-2} \left( q  - {n +1 \choose 2} \right)- \left( C_0 +{n +1 \choose 2}  C_1 \right)
\end{align*}
Let $x$ denote the first summand $q^{n-2} \left( q  - {n +1 \choose 2} \right)$ and $y$ denote the second summand $C_0 +{n +1 \choose 2}  C_1 $.
To find a lower bound for the above expression, we need to find an upper bound on $y$ for fixed $k$ and $n$. For the first term, note that since $2n^2  \geq {n+1 \choose 2}$
we require $q > 2n^2$.
With this choice of $q$, an upper bound of the second term is obtained using the bounds on $C_0$ and $C_1$ obtained above:
\begin{align*}
y&\leq  k^{2n} \sqrt{q^n}+ {n +1 \choose 2} k^{2n} \sqrt{q^{n-1}} \leq k^{2n} q^{n/2} (1 + 2n^2)\\ &\leq k^{2n} q^{n/2}(1 + q) \leq k^{2n} q^{(n/2) + 2}
\end{align*}
Since we want $x > y$, it suffices to choose $q$ such that $q > 2n^2$ and $q^{n-2} \geq  k^{2n} q^{(n/2) + 2}$. From the second inequality it follows that we require $q^{n-8} > k^{4n}$ and so it suffices to have $q > k^{4n/(n-8)}$. Thus it suffices to define $C(k,n) := \textrm{max} \left( 2n^2,  k^{4n/(n-8)} \right)$.
\end{proof}

\noindent The special case of Theorem \ref{weil_thm} with $n =2$ is required to write the Jordan block $J_{\alpha, n}$ where $\alpha$ is a primitive element of $\F_q$ as a sum of two $k$th powers. We will need  two solutions $(a,b)$ and $(c,d)$ in $\F_q \times \F_q$ are required to $X_1^k + X_2^k = \lambda$ such that $a^k \neq c^k$ and $b^k \neq d^k$ for any fixed $k \geq 2$.

\begin{prop}\label{non_nilpotent_case}
Let $k \geq 2$, and the characteristic of $\F_q$ not $2$. Let $\lambda$ be a nonzero element in $\F_q$. Consider the equation 
\begin{equation}\label{2_equation}
X_1^k  + X_2^k = \lambda.
\end{equation}
For all  $q > k^{16}$, there exists two solutions $(a,b)$ and $(c,d)$  to \ref{2_equation} in $\F_q \times \F_q$ such that $a^k \neq c^k$ and $b^k \neq d^k$.
\end{prop}

\begin{proof}
Let $a_1 =1/\lambda$ and $a_2 = 1/\lambda$ in Theorem \ref{weil_thm}, then the number of solutions $N$ to the equation \eqref{2_equation} satisfies the following inequality 
$
\abs{N - q} \leq k^4 \sqrt{q}.
$
 The number of solutions with fixed $X_1^k \in \F_q$ is at most $k^2$ and similarly with fixed $X_2^k$ is at most $k^2$. Thus for a pair of solutions $(a,b)$ and $(c,d)$ with $a^k \neq c^k$ and $b^k \neq d^k$ to exist we require $q - k^4 \sqrt{q} \geq 2k^2 + 1$. Since $k \geq 2$ it suffices to have $\sqrt{q} - k^4 > 0$ and $\sqrt{q} \geq k^4 \geq 2k^2 +1$. Thus if $q > k^{16}$ we have the existence of two solutions with the desired properties.
\end{proof}

\section{A Determinant Computation}\label{det_computation}

\begin{lem}\label{ele_com}
Let $\F_q$ be a finite field of characteristic not $2$, and $n$ be an integer $\geq 3$. Let $B_n$ be the matrix with the last row $[-y_1 \; -y_2 \; \ldots -y_{n-1} \; -1 ]$ and the last column $[z_1 \; z_2 \; \ldots \;  z_{n-2} \; 0 \; -1]^T$ and with $1$s along the super-diagonal except at the $((n-1), n)$-th position where it is $0$. Then, the characteristic polynomial of $B_n$ (in $\F_p[t]$) is
\begin{align*}
&t^n + t^{n-1} \\
&+ (y_1 z_1 + y_2 z_2 + \ldots + y_{n-2} z_{n-2}) t^{n-2}  \\
&+ (y_1 z_2 + y_2 z_3 + \ldots + y_{n-3} z_{n-2}) t^{n-3} \\
&+ \ldots \\
&+ (y_1 z_{n-3}  + y_2 z_{n-2} ) t^{2} \\
&+ (y_1 z_{n-2} ) t.
\end{align*}
\end{lem} 

\begin{proof}
The characteristic polynomial of $B_n$ is the determinant of the following matrix: 
$$
\begin{bmatrix}
t & -1 & 0 & 0 & \ldots & 0 &  0 &-z_1 \\
0 & t & -1 & 0 & \ldots & 0 &  0 & -z_2 \\
0 & 0 & t & -1 & \ldots & 0 & 0 & -z_3 \\
\vdots & \vdots & \vdots & \vdots & \ldots & \vdots  & \vdots &\vdots \\
0 & 0 & 0 & 0 & \ldots & t & -1 & - z_{n-2} \\
0 & 0 & 0 & \ldots & \ldots & 0 & t & 0 \\
y_1 &y_2 & y_3 & \ldots & \ldots & y_{n-2} & {y_{n-1}} & t +1 
\end{bmatrix}
$$

Expanding with respect to the first column, the determinant is the sum $t  \det B_{n-1} + (-1)^{n+1} y_1  \det D_{n-1}$, where {\Small
$$
B_{n-1} = 
\begin{bmatrix}
t & -1 & 0 & \ldots & 0 &  0 & -z_2 \\
0 & t & -1 & \ldots & 0 & 0 & -z_3 \\
\vdots & \vdots & \vdots & \ldots & \vdots  & \vdots &\vdots \\
0 & 0 & 0 & \ldots & t & -1 & - z_{n-2} \\
0 & 0 & \ldots & \ldots & 0 & t & 0 \\
y_2 & y_3 & \ldots & \ldots & y_{n-2} & {y_{n-1}} & t +1 
\end{bmatrix}, 
D_{n-1} = \begin{bmatrix}
 -1 & 0 & 0 & \ldots & 0 &  0 &-z_1 \\
 t & -1 & 0 & \ldots & 0 &  0 & -z_2 \\
 0 & t & -1 & \ldots & 0 & 0 & -z_3 \\
\vdots & \vdots & \vdots & \ldots & \vdots  & \vdots &\vdots \\
 0 & 0 & 0 & \ldots & t & -1 & - z_{n-2} \\
0 & 0 & \ldots & \ldots & 0 & t & 0 \\
\end{bmatrix}
$$}

By induction on the size $n$ of $B_n$ beginning with the base case as $n =3$, it follows that the determinant of $B_{n-1}$ is 
\begin{align*}
&t^{n-1} + t^{n-2} + (y_2 z_2 + y_3 z_3 + \ldots + y_{n-3} z_{n-3}) t^{n-3} + \ldots + (y_2 z_{n-3}  + y_2 z_{n-2} ) t^{2} + (y_2 z_{n-2} ) t.
\end{align*}
As for the determinant of $D_{n-1}$ 
let us consider the notation $R_i \mapsto R_i + \alpha R_j$ denotes that row $R_i$ is changed by adding to it $\alpha$ times $R_j$. Now by means of elementary row operations (in order) $R_2 \mapsto R_2 + t R_1$, then $R_3 \mapsto R_3 + t R_2$ (where we have to use the altered second row in the previous step) \ldots, $R_{n-1} \mapsto R_{n-1} + t R_{n-2}$ we obtain the following matrix. Abusing notation we denote the resulting matrix again by $D_{n-1}$ which is of the form
$$
\begin{bmatrix}
 -1 & 0 & 0 & \ldots & 0 & -z_{1}\\
 0 & -1 & 0 & \ldots & 0 & -z_{2} - t z_1 \\
 0 & 0& -1 & \ldots & 0 & - z_3 - z_2 t -z_1 t^2 \\
\vdots & \vdots & \vdots & \ldots & \vdots &\vdots \\
0 & 0 & 0 & \ldots & 0 & -tz_{n-2} - t^2 z_{n-3}- \ldots - t^{n-2} z_1\\
\end{bmatrix}
$$
whose determinant is clearly $(-1)^{n-1} (tz_{n-2} + t^2 z_{n-3} + \ldots + t^{n-2} z_1)$.  It follows that the determinant of $D_{n-1}$ is $(-1)^{n-1} (tz_{n-2} + t^2 z_{n-3} + \ldots + t^{n-2} z_1)$. Finally, as we already note the the determinant of $B_n$ is $t \times \det B_{n-1} + (-1)^{n+1} y_1 (-1)^{n-1} (tz_{n-2} + t^2 z_{n-3} + \ldots + t^{n-2} z_1)$ which has the desired form in the statement.
\end{proof}

\section*{Acknowledgements}
\noindent It is a great pleasure to thank Michael Larsen for suggesting the question and for many useful discussions. We thank the referee(s) for their careful reading and suggestions.

\bibliographystyle{elsarticle-num} 

\begin{thebibliography}{00}

\bibitem{Ki} Krishna Kishore: Matrix Waring Problem, Linear Algebra and its Applications, Vol. 646, (2022), 84--94. 

\bibitem{KS} R. Kundu, A. Singh: Generating functions for the powers in $GL(n,q)$: accepted, Israel Journal of Mathematics (2022).


\bibitem{LST}  M.Larsen, A. Shalev, P.H. Tiep: The Waring problem for finite simple groups, Annals of Mathematics, Vol. 174, Issue 3, (2011), 1885--1950

\bibitem{LW} S. Lang , A. Weil: Number of points of varieties in finite fields: American Journal of Mathematics, Vol. 76, No. 4, (1954), 819--827.


\bibitem{Sc}
W. M. Schmidt: Equations over finite fields: an elementary approach, Lecture Notes in Mathematics, 536. New York: Springer-Verlag, 1976.

\bibitem{Sh} Shalev, Aner: Word maps, conjugacy classes, and a noncommutative Waring-type theorem, Ann. of Math. (2) 170 (2009), no. 3, 1383–-1416.

\bibitem{Sm}
C. Small: Sums of powers in large finite fields, Proc. Amer. Math. Soc. 65 (1977), 35--36.


\end{thebibliography}

\end{document}